\documentclass[a4,11pt]{article}
\usepackage{authblk}
\usepackage{layout}
\usepackage{here}
\usepackage{amsmath}
\usepackage{amsthm,amssymb}
\usepackage{amsfonts}

\usepackage{graphicx}

\newtheorem{theorem}{Theorem}[section]
\newtheorem{prop}{Lemma}[section]

\newtheorem{remark}{Remark}[section]
\usepackage{color}
\usepackage{cancel,cases}
\usepackage{comment}
\usepackage{cite}
\usepackage{float}
\usepackage{color}

\DeclareMathOperator{\esssup}{ess\,sup}

\begin{document}

\begin{center}
\noindent{\bf\Large Estimation of Sobolev embedding constant on a domain dividable into bounded convex domains
}\vspace{10pt}\\
 {\normalsize Makoto Mizuguchi$^{1,*}$, Kazuaki Tanaka$^{1,\dagger}$, Kouta Sekine$^{2,\ddagger}$, Shin'ichi Oishi$^{2,\S}$}\vspace{5pt}\\
 {\it\normalsize $^{1}$Graduate School of Fundamental Science and Engineering, Waseda University, Japan\\
 $^{2}$Faculty of Science and Engineering, Waseda University, Japan}
\end{center}

{\bf Abstract}.
This paper is concerned with an explicit value of the embedding constant from $W^{1,q}(\Omega)$ to $L^{p}(\Omega)$ for a bounded domain $\Omega\subset\mathbb{R}^N~(N\in\mathbb{N})$, where $1\leq q\leq p\leq \infty$. To obtain this value, we previously proposed a formula for estimating the embedding constant on bounded and unbounded Lipschitz domains by estimating the norm of Stein's extension operator, in the article (K. Tanaka, K. Sekine, M. Mizuguchi, and S. Oishi, Estimation of Sobolev-type embedding constant on domains with minimally smooth boundary using extension operator, Journal of Inequalities and Applications, Vol.~$389$, pp.~$1$-$23$, 2015). This formula is also applicable to a domain that can be divided into Lipschitz domains. However, the values computed by the previous formula are very large. In this paper, we propose several sharper estimations of the embedding constant on a bounded domain that can be divided into convex domains. 
 
\renewcommand{\thefootnote}{\fnsymbol{footnote}}
\footnote[0]{{\it E-mail address:} $^{*}$\texttt{makoto.math@fuji.waseda.jp},~~
$^{\dagger}$\texttt{imahazimari@fuji.waseda.jp},\\[2pt]
\hspace*{2.95cm}$^{\ddagger}$\texttt{k.sekine@aoni.waseda.jp},~~
$^{\S}$\texttt{oishi@waseda.jp}\\[-3pt]}
\renewcommand\thefootnote{*\arabic{footnote}}




\section{Introduction}
We consider the Sobolev type embedding constant $C_p(\Omega)$ from $W^{1,q}(\Omega)~(1\leq q\leq p\leq\infty)$ to $L^p(\Omega)$.
The constant $C_p(\Omega)$ satisfies
\begin{align}\label{aim}
\left(\int_{\Omega}|u(x)|^pdx\right)^{\frac{1}{p}}\leq C_p(\Omega)\left(\int_{\Omega}|u(x)|^qdx+\int_{\Omega}|\nabla u(x)|^qdx\right)^{\frac{1}{q}}
\end{align}
for all $u\in W^{1,q}(\Omega)$, where $\Omega\subset\mathbb{R}^N~(N\in\mathbb{N})$ is a bounded domain and $|x|=\sqrt{\sum_{j=1}^{N}x_j^2}$ for $x=(x_1,\cdots, x_N)\in\mathbb{R}^N$.
Here, $L^p(\Omega)~(1\leq p<\infty)$ is the functional space of the $p$th power Lebesgue integrable functions over $\Omega$ endowed with the norm $\|f\|_{L^p(\Omega)}:=(\int_{\Omega}|f(x)|^pdx)^{1/p}$ for $f\in L^p(\Omega)$, and $L^{\infty}(\Omega)$ is the functional space of Lebesgue measurable functions over $\Omega$ endowed with the norm $\|f\|_{L^{\infty}(\Omega)}=\esssup\displaylimits_{x\in\Omega}|f(x)|$ for $f\in L^{\infty}(\Omega)$. Moreover, $W^{k,p}(\Omega)$ is the $k$th order $L^p$-Sobolev space on $\Omega$ endowed with the norm $\|f\|_{W^{1,p}(\Omega)}=(\int_{\Omega}|f(x)|^pdx+\int_{\Omega}|\nabla f(x)|^pdx)^{1/p}$ for $f\in W^{1,p}(\Omega)$ if $1\leq p<\infty$ and $\|f\|_{W^{1,\infty}(\Omega)}=\esssup\displaylimits_{x\in\Omega}|f(x)|+\esssup\displaylimits_{x\in\Omega}|\nabla f(x)|$ for $f\in W^{1,\infty}(\Omega)$ if $p=\infty$.

Since inequality \eqref{aim} has significance for studies on partial differential equations, many researchers studied this type of Sobolev inequality and an explicit value of $C_p(\Omega)$ (see, e.g., \cite{S.L.sobolev, Adams, Whitney, Hestenes, Calderon, Stein, Roger}), following the pioneering work by S.L.~Sobolev \cite{S.L.sobolev}.
In particular, our interest is in the applicability of the constant to verified numerical computation methods for PDEs, which originate from Nakao's \cite{nakao} and Plum's work \cite{plum} and have been further developed by many researchers (see, e.g., \cite{nakao, plum, cai} and the references therein).

The existence of $C_p(\Omega)$ for various classes of domains $\Omega$ (e.g., domains with the cone condition, domains with the Lipschitz boundary, and the $(\varepsilon, \delta)-$domain) has been proven by constructing suitable extension operators from $W^{k,p}(\Omega)$ to $W^{k,p}(\mathbb{R}^N)$ (see, e.g.,\cite{Whitney, Hestenes, Calderon, Stein, Roger}).

Several formulas for computing  explicit values of $C_p(\Omega)$ have been proposed.
For example, the best constant in the classical Sobolev inequality on $\mathbb{R}^N$ was independently shown by Aubin \cite{Aubin} and Talenti \cite{Talenti}.
Moreover, for the case in which $N=1$ and $p=\infty$, the best constant of $C_p(\Omega)$ was proposed under suitable boundary conditions, e.g., the Dirichlet, the Neumann, and the Periodic condition \cite{1dim1, 1dim2, 1dim3, 1dim4, 1dim5}. In recent years, several formulas for obtaining an explicit value of $C_p(\Omega)$ for $\Omega\subset \mathbb{R}^N~(N\geq2)$ have been further proposed. For example, for a square domain $\Omega$, a tight estimate of $C_p(\Omega)$ was provided in \cite{cai}. Moreover, we have previously proposed a formula for computing an explicit value of $C_p(\Omega)$ for (bounded and unbounded) Lipschitz domains $\Omega$ by estimating the norm of Stein's extension operator \cite{tanaka}. This formula can be applied to domains $\Omega$ that can be divided into a finite number of Lipschitz domains $\Omega_i~(i=1,2,3,\cdots, n)$ such that 
\begin{align}\label{omega1}
\overline{\Omega}=\displaystyle\bigcup_{1\leq i\leq n}\overline{\Omega_i}
\end{align}
and
\begin{align}\label{omega2}
\Omega_i\cap\Omega_j=\phi~(i\neq j),
\end{align}
where $\phi$ is the empty set and $\overline{\Omega}$ denotes the closure of $\Omega$ (see Theorem \ref{estimatecor}). Although this formula is applicable to such general domains, the values computed by this formula are very large; see Section \ref{sec:num} for some explicit values. 

In this paper, we propose sharper estimations of $C_p(\Omega)$ for a domain $\Omega$ that can be divided into a finite number of bounded convex domains $\Omega_i~(i=1,2,3,\cdots, n)$ satisfying \eqref{omega1} and \eqref{omega2}. Note that the present class of $\Omega$ is somewhat special compared to the class treated in \cite{tanaka}, since any bounded convex domain is a Lipschitz domain. 
To obtain a sharper estimation of $C_p(\Omega)$, we focus on the constants $D_p(\Omega)$, which lead to the explicit values of $C_p(\Omega)$ (see Theorem \ref{mainembedding}), such that
\begin{align}\label{poincareine}
\left(\int_{\Omega}|u(x)-u_{\Omega}|^pdx\right)^{\frac{1}{p}}\leq D_p(\Omega)\left(\int_{\Omega}|\nabla u(x)|^qdx\right)^{\frac{1}{q}}~~\mbox{for all}~ u\in W^{1,q}(\Omega),
\end{align}
where $u_{\Omega}=|\Omega|^{-1}\int_{\Omega}u(x)dx$ and $|\Omega|$ is the measure of $\Omega$. 
 Inequality \eqref{poincareine} is called the Sobolev-Poincar\'e inequality, which has also been studied by many researchers (see, e.g., \cite{koonj, hajkos, bojarski, denylion}).
The existence of $D_p(\Omega)$ was shown for a John domain $\Omega$ while assuming that $1\leq q<N$, $p=Nq/(N-q)$ \cite{bojarski}.
It was also shown that, when $p\neq Nq/(N-q)$, $D_p(\Omega)$ exists if and only if $W^{1,q}(\Omega)$ is continuously embedded into $L^p(\Omega)$ \cite{denylion}.
Moreover, there are several formulas for deriving an explicit value of $D_p(\Omega)$ for one-dimensional domains $\Omega$ \cite{BenSai, LiuWhe, ChuWhe}. 
In the higher-dimensional cases, however, little is known about explicit values of $D_p(\Omega)$, except for some special cases (see, e.g., \cite{gabriel} and \cite{payne} for the cases in which $p=q=1$ and $p=q=2$, respectively).

We propose four theorems (Theorem \ref{conesti0} to \ref{conesti1}) for obtaining explicit values of $D_p(\Omega)$ on a bounded convex domain $\Omega$.
Each theorem can be used under the corresponding conditions listed in Table \ref{tb:assume0}.
\begin{table}[htb]
\begin{center}
\caption{The assumptions of $p$, $q$, and $N$ imposed in Theorem \ref{conesti0}, \ref{conesti00}, \ref{conesti}, and \ref{conesti1}}
\vspace{2mm}
\begin{tabular}{c|c|c|c}
Theorem & $p$ & $q$ & $N$ \\ \hline
\ref{conesti0}&$2<p\leq \frac{2N}{N-1}$&$q\geq\frac{p}{p-1}$& $N\geq 1$\\ \hline
\ref{conesti00}&$2<p<\frac{2N}{N-2}$&$q=2$& $N\geq 2$\\ \hline
\ref{conesti}&$q\leq p<\frac{qN}{N-q}$&$q\geq1$&$N\geq q$\\ \hline
\ref{conesti1}&$p=\infty$&$q\geq1$& $N<q$\\   
\end{tabular}
\label{tb:assume0}
\end{center}
\end{table}
In Theorem \ref{conesti0} and \ref{conesti00}, formulas of $D_p(\Omega)$ are derived from the best constant in the Hardy-Littlewood-Sobolev inequality on $\mathbb{R}^N$.
In Theorem \ref{conesti} and \ref{conesti1}, formulas of $D_p(\Omega)$ are also derived from the best constant in Young's inequality on $\mathbb{R}^N$. These values of $D_p(\Omega)$ yield the explicit values of $C_p(\Omega)$ through Theorem \ref{mainembedding}.

The remainder of this paper is organized as follows:
In Section \ref{sec:preliminary}, we introduce the notation used throughout this paper.
In Section \ref{subsec1}, we propose a formula for deriving $C_p(\Omega)$ from a known value of $D_p(\Omega)$.
In Section \ref{subsec2}, we propose four formulas for obtaining the explicit value of $D_p(\Omega)$ under the conditions listed in Table \ref{tb:assume0} for a bounded convex domain $\Omega$.
In Section \ref{sec:num}, explicit values of $C_p(\Omega)$ are estimated for certain domains.

\section{Notation}\label{sec:preliminary}
For any bounded domain $S\subset\mathbb{R}^N~(N\in\mathbb{N})$, we define $d_{S}$:=$\sup_{x,y\in S}|x-y|$.
The closed ball centered around $z\in\mathbb{R}^N$ with radius $\rho>0$ is denoted by $B(z,\rho):=\{x\in \mathbb{R}^N\mid|x-z|\leq \rho\}$. For $m\geq 1$, let $m'$ be H\"{o}lder's conjugate of $m$, that is, $m'$ is defined by
\begin{align*}
\begin{cases} 
m'=\infty,~~~~~~~\mbox{if}~m=1,\\
m'=\frac{m}{m-1},~~~~\mbox{if}~1<m<\infty,\\ 
m'=1,~~~~~~~~\mbox{if}~m=\infty.
\end{cases} 
\end{align*}

For two domains $\Omega\subseteq\mathbb{R}^N$ and $\Omega'\subseteq\mathbb{R}^N$ such that $\Omega\subseteq\Omega'$, we define the operator $E_{\Omega,\Omega'}:L^p(\Omega)\to L^p(\Omega')~(1\leq p\leq\infty)$ by
\begin{align*}
\left(E_{\Omega,\Omega'}f\right )(x)=
\begin{cases}
f(x),~~x\in \Omega,\\
~~0,\hspace{0.6cm}x\in\Omega'\setminus\Omega
\end{cases} 
\end{align*}
for $f\in L^{p}(\Omega)$. Note that $E_{\Omega,\Omega'}f\in L^p(\Omega')$ satisfies
\[
\|E_{\Omega,\Omega'}f\|_{L^{p}(\Omega')}=\|f\|_{L^{p}(\Omega)}.
\]

\section{Formula for explicit values of the embedding constant}\label{subsec1}
The following theorem enables us to obtain an explicit value of $C_p(\Omega)$ from a known $D_p(\Omega)$. 
\begin{theorem}\label{mainembedding} 
Let $\Omega\subset\mathbb{R}^N~(N\in\mathbb{N})$ be a bounded domain, and let $p$, $q\in [1,\infty]$. Suppose that there exists a finite number of bounded domains $\Omega_i~(i=1,2,3,\cdots, n)$ satisfying \eqref{omega1} and \eqref{omega2}.
Moreover, suppose that for every $\Omega_i~(i=1,2,3,\cdots, n)$ there exist constants $D_p(\Omega_i)$ such that
\begin{align}\label{mainine}
\|u-u_{\Omega_i}\|_{L^{p}(\Omega_i)}\leq D_{p}(\Omega_i)\|\nabla u\|_{L^q(\Omega_i)} ~\mbox{for all}~u\in W^{1,q}(\Omega_i).
\end{align}
Then, \eqref{aim} holds valid for
\begin{align}
C_p(\Omega)=2^{1-\frac{1}{q}}\max\left(\max_{1\leq i\leq n}|\Omega_i|^{\frac{1}{p}-\frac{1}{q}},~\max_{1\leq i\leq n}D_p(\Omega_i)\right), \label{sec3:theo:cp}
\end{align}
where this formula is understood with $1/\infty=0$ when $p=\infty$ and/or $q=\infty$.
\end{theorem}
\begin{proof}
Since every $\Omega_i$ is bounded, H\"{o}lder's inequality states that 
\begin{align}\label{matomeholder}\nonumber
\|u_{\Omega_i}\|_{L^p(\Omega_i)}&=\left|\int_{\Omega_i}|\Omega_i|^{-1}u(y)dy\right|\|1\|_{L^p(\Omega_i)}\\ \nonumber
&\leq|\Omega_i|^{-1+\frac{1}{q'}}\|u\|_{L^q(\Omega_i)}|\Omega_i|^{\frac{1}{p}}\\ 
&=|\Omega_i|^{\frac{1}{p}-\frac{1}{q}}\|u\|_{L^q(\Omega_i)},
\end{align}
where we again assume $1/\infty=0$.

In the following proof, we consider both of the cases in which $p=\infty$ and $p<\infty$.
When $p=\infty$, we have
\begin{align}\nonumber
\|u\|_{L^\infty(\Omega)}&=\max_{1\leq i\leq n}\|u\|_{L^{\infty}(\Omega_i)}\\ \nonumber
&\leq\max_{1\leq i\leq n}\left(\|u_{\Omega_i}\|_{L^\infty(\Omega_i)}+\|u-u_{\Omega_i}\|_{L^\infty(\Omega_i)}\right).
\end{align}
From \eqref{mainine} and \eqref{matomeholder}, it follows that
\begin{align*}
&\|u\|_{L^\infty(\Omega)}\\ \nonumber
&\leq \max_{1\leq i\leq n}\left(|\Omega_i|^{-\frac{1}{q}}\|u\|_{L^q(\Omega_i)}+D_{\infty}(\Omega_i)\|\nabla u\|_{L^q(\Omega_i)}\right)\\
&\leq \max\left\{\max_{1\leq i\leq n}|\Omega_i|^{-\frac{1}{q}}, \max_{1\leq i\leq n}D_{\infty}(\Omega_i)\right\}\max_{1\leq i\leq n}\left(\|u\|_{L^q(\Omega_i)}+\|\nabla u\|_{L^q(\Omega_i)}\right).
\end{align*}
This implies the case in which $q=\infty$.
For $q<\infty$, we have
\begin{align*}
&\|u\|_{L^\infty(\Omega)}\\
\leq &\max\left\{\max_{1\leq i\leq n}|\Omega_i|^{-\frac{1}{q}}, \max_{1\leq i\leq n}D_{\infty}(\Omega_i)\right\}\left(\sum_{1\leq i\leq n}\left(\|u\|_{L^q(\Omega_i)}+\|\nabla u\|_{L^q(\Omega_i)}\right)^{q}\right)^{\frac{1}{q}}\\
\leq &2^{1-\frac{1}{q}}\max\left\{\max_{1\leq i\leq n}|\Omega_i|^{-\frac{1}{q}}, \max_{1\leq i\leq n}D_{\infty}(\Omega_i)\right\}\|u\|_{W^{1,q}(\Omega)},
\end{align*}
where the last inequality follows from $(s+t)^q\leq 2^{q-1}(s^q+t^q)$ for $s,t\geq 0$.

When $p<\infty$, we have 
\begin{align*}\nonumber
\|u\|_{L^p(\Omega)}&=\left(\sum_{1\leq i\leq n}\int_{\Omega_i}|u(y)|^pdy\right)^{\frac{1}{p}}\\ \nonumber
&=\left(\sum_{1\leq i\leq n}\|u\|_{L^p(\Omega_i)}^p\right)^{\frac{1}{p}}\\ 
&\leq\left(\sum_{1\leq i\leq n}\left(\|u_{\Omega_i}\|_{L^p(\Omega_i)}+\|u-u_{\Omega_i}\|_{L^p(\Omega_i)}\right)^p\right)^{\frac{1}{p}}.
\end{align*}
From \eqref{mainine} and \eqref{matomeholder}, it follows that
\begin{align*}\nonumber
\|u\|_{L^p(\Omega)}&\leq\left(\sum_{1\leq i\leq n}\left(|\Omega_i|^{\frac{1}{p}-\frac{1}{q}}\|u\|_{L^q(\Omega_i)}+D_{p}(\Omega_i)\|\nabla u\|_{L^q(\Omega_i)}\right)^p\right)^{\frac{1}{p}}\\ \nonumber
&\leq\left(\sum_{1\leq i\leq n}\left(|\Omega_i|^{\frac{1}{p}-\frac{1}{q}}\|u\|_{L^q(\Omega_i)}+D_{p}(\Omega_i)\|\nabla u\|_{L^q(\Omega_i)}\right)^q\right)^{\frac{1}{q}}\\
&\leq 2^{1-\frac{1}{q}}\left(\sum_{1\leq i\leq n}\left(|\Omega_i|^{\frac{q}{p}-1}\|u\|_{L^q(\Omega_i)}^q+D_{p}(\Omega_i)^q\|\nabla u\|_{L^q(\Omega_i)}^q\right)\right)^{\frac{1}{q}}.
\end{align*}
Therefore, we obtain
\begin{align*}
\|u\|_{L^p(\Omega)}\leq 2^{1-\frac{1}{q}}\max\left\{\max_{1\leq i\leq n}|\Omega_i|^{\frac{1}{p}-\frac{1}{q}}, \max_{1\leq i\leq n}D_i(\Omega_i)\right\}\|u\|_{W^{1,q}(\Omega)}.
\end{align*}

\end{proof}

\section{Bounds for $D_p(\Omega_i)$}\label{subsec2}
We denote the gamma function by $\Gamma$ (i.e., $\Gamma(x)=\int_{0}^{\infty}t^{x-1}e^{-t}dt$~for~$x>0$). For $f\in L^r(\mathbb{R}^N)$ and $g\in L^s(\mathbb{R}^N)~(1\leq r,s\leq\infty)$, $f*g: \mathbb{R}^N\to\mathbb{R}$ is the convolution of $f$ and $g$ defined by
\begin{align*}
(f*g)(x):=\int_{\mathbb{R}^N}f(x-y)g(y)dy\left(=\int_{\mathbb{R}^N}f(x)g(x-y)dy\right).
\end{align*}
In the following three lemmas, we show some existing results required to obtain explicit values of $D_p(\Omega_i)$ in \eqref{mainine} for bounded convex domains $\Omega_i$.

\begin{prop}[see, e.g.,\hspace{1mm}\cite{heatkernel}]\label{eq1}
Let $\Omega\subset\mathbb{R}^N~(N\in\mathbb{N})$ be a bounded convex domain. For $u\in W^{1,1}(\Omega)$ and any point $x\in \Omega$, we have
\begin{align*}
|u(x)-u_{\Omega}|\leq \frac{d_{\Omega}^N}{N|\Omega|}\int_{\Omega}|x-y|^{1-N}|\nabla u(y)| dy.
\end{align*}
\end{prop}

\begin{prop}[Hardy-Littlewood-Sobolev's inequality \cite{hls}]\label{th0}
For $\lambda>0$, we put $h_{\lambda}(x):=|x|^{-\lambda}$.  If $0<\lambda<N$,  
\begin{align}\label{hlss1}
\|h_{\lambda}*g\|_{L^{\frac{2N}{\lambda}}(\mathbb{R}^N)}\leq C_{\lambda, N}\|g\|_{L^{\frac{2N}{2N-\lambda}}(\mathbb{R}^N)}~\mbox{for all}~g\in L^{\frac{2N}{2N-\lambda}}(\mathbb{R}^N)
\end{align}
holds valid for 
\begin{align}\label{hlsesti1}
C_{\lambda, N}=\pi^{\frac{\lambda}{2}}\frac{\Gamma(\frac{N}{2}-\frac{\lambda}{2})}{\Gamma(N-\frac{\lambda}{2})}\left(\frac{\Gamma(\frac{N}{2})}{\Gamma(N)}\right)^{-1+\frac{\lambda}{N}},
\end{align}
where this is the best constant in \eqref{hlss1}.

Moreover, if $N<2\lambda<2N$, 
\begin{align}\label{hlss2}
\|h_{\lambda}*g\|_{L^{\frac{2N}{2\lambda-N}}(\mathbb{R}^N)}\leq \tilde C_{\lambda, N}\|g\|_{L^{2}(\mathbb{R}^N)}~\mbox{for all}~g\in L^2(\mathbb{R}^N)
\end{align}
holds valid for 
\begin{align}\label{hlsesti2}
\tilde C_{\lambda, N}=\pi^{\frac{\lambda}{2}}\frac{\Gamma(\frac{N}{2}-\frac{\lambda}{2})}{\Gamma(\frac{\lambda}{2})}\sqrt{\frac{\Gamma(\lambda-\frac{N}{2})}{\Gamma(\frac{3N}{2}-\lambda)}}\left(\frac{\Gamma(\frac{N}{2})}{\Gamma(N)}\right)^{-1+\frac{\lambda}{N}},
\end{align}
where this is the best constant in \eqref{hlss2}.
\end{prop}

\begin{prop}[Young's inequality~\cite{young}]\label{th1}
Suppose that $1\leq t,r,s\leq\infty$ and $1/t=1/r+1/s-1\geq0$. For $f\in L^{r}(\mathbb{R}^N)$ and $g\in L^s(\mathbb{R}^N)$, we have
\begin{align}\label{young}
\|f*g\|_{L^t(\mathbb{R}^N)}\leq (A_{r}A_{s}A_{t'})^N\|f\|_{L^{r}(\mathbb{R}^N)}\|g\|_{L^{s}(\mathbb{R}^N)}
\end{align}
with
\begin{align*}
A_m=
\begin{cases}
\sqrt{m^{\frac{2}{m}-1}(m-1)^{1-\frac{1}{m}}}~~(1<m<\infty),\\[2mm]
\hspace{1.5cm}1\hspace{1.8cm}(m=1,~\infty).
\end{cases}
\end{align*} 
The constant $(A_{r}A_{s}A_{t'})^N$ is in fact the best constant in \eqref{young}.
\end{prop}

The following Theorems \ref{conesti0}, \ref{conesti00}, \ref{conesti}, and \ref{conesti1} provide estimations of $D_p(\Omega)$ for a bounded convex domain $\Omega$, where $p$, $q$, and $N$ are imposed on the assumptions listed in Table \ref{tb:assume0}.

\begin{theorem}\label{conesti0}
Let $\Omega\subset\mathbb{R}^N~(N\in\mathbb{N})$ be a bounded convex domain. Assume that $p\in\mathbb{R}$ satisfies $2<p\leq 2N/(N-1)$ if $N\geq2$ and $2<p<\infty$ if $N=1$. 
For $q\in\mathbb{R}$ such that $q\geq p/(p-1)$, we have
\begin{align*}
\|u-u_{\Omega}\|_{L^p(\Omega)}\leq D_p(\Omega)\|\nabla u\|_{L^{q}(\Omega)}~\mbox{for~all}~u\in W^{1,q}(\Omega)
\end{align*}
with
\begin{align*}
D_p(\Omega)=\frac{{d_\Omega}^{1+\frac{2N}{p}}\pi^{\frac{N}{p}}}{N|\Omega|^{\frac{p}{q(p-1)}}}\frac{\Gamma(\frac{p-2}{2p}N)}{\Gamma(\frac{p-1}{p}N)}\left(\frac{\Gamma(N)}{\Gamma(\frac{N}{2})}\right)^{\frac{p-2}{p}}.
\end{align*}

\end{theorem}

\begin{proof}
Let $u\in W^{1,q}(\Omega)$. Since $p\leq 2N/(N-1)$ and $1-N+(2N/p)\geq0$, it follows that $|x-z|^{1-N+\frac{2N}{p}}\leq d_{\Omega}^{1-N+\frac{2N}{p}}$ for $x, z\in\Omega$. Lemma \ref{eq1} implies that, for a fixed $x\in \Omega$, 
\begin{align*}\nonumber
|u(x)-u_{\Omega}|&\leq\frac{d_{\Omega}^{N}}{N|\Omega|}\int_{\Omega}|x-z|^{1-N+\frac{2N}{p}}|x-z|^{-\frac{2N}{p}}|\nabla u(z)|dz\\ \nonumber
&\leq\frac{d_{\Omega}^{1+\frac{2N}{p}}}{N|\Omega|}\int_{\Omega}|x-z|^{-\frac{2N}{p}}|\nabla u(z)|dz\\
&\leq \frac{d_{\Omega}^{1+\frac{2N}{p}}}{N|\Omega|}\int_{\mathbb{R}^N}|x-z|^{-\frac{2N}{p}}\left(E_{\Omega,\mathbb{R}^N}|\nabla u|\right)(z)dz.
\end{align*}
Therefore, 
\begin{align*}
\|u-u_{\Omega}\|_{L^p(\Omega)}&\leq\frac{{d_\Omega}^{1+\frac{2N}{p}}}{N|\Omega|}\left(\int_{\Omega}\left(\int_{\mathbb{R}^N}|x-z|^{-\frac{2N}{p}}\left(E_{\Omega,\mathbb{R}^N}|\nabla u|\right)(z)dz\right)^pdx\right)^{\frac{1}{p}}\\
&\leq \frac{{d_\Omega}^{1+\frac{2N}{p}}}{N|\Omega|}\left(\int_{\mathbb{R}^N}\left(\int_{\mathbb{R}^N}|x-z|^{-\frac{2N}{p}}\left(E_{\Omega,\mathbb{R}^N}|\nabla u|\right)(z)dz\right)^pdx\right)^{\frac{1}{p}}.
\end{align*}
Since $q\geq p/(p-1)$ and $\Omega$ is bounded, we have $|\nabla u|\in L^{p/(p-1)}(\Omega)$. Therefore,  
\begin{align*}
\|u-u_{\Omega}\|_{L^p(\Omega)}&\leq\frac{d_{\Omega}^{1+\frac{2N}{p}}}{N|\Omega|}C_{\frac{2N}{p}N, N}\|E_{\Omega,\mathbb{R}^N}|\nabla u|\|_{L^{\frac{p}{p-1}}(\mathbb{R}^N)}\\
&=\frac{d_{\Omega}^{1+\frac{2N}{p}}}{N|\Omega|}C_{\frac{2N}{p}N, N}\|\nabla u\|_{L^{\frac{p}{p-1}}(\Omega)},
\end{align*}
where $C_{\frac{2N}{p}N, N}$ is defined in \eqref{hlsesti1} with  $\lambda=2N/p$. Since $q\geq p/(p-1)$, H\"{o}lder's inequality moreover implies 
\begin{align*}
\|u-u_{\Omega}\|_{L^p(\Omega)}&\leq\frac{d_{\Omega}^{1+\frac{2N}{p}}}{N|\Omega|^{\frac{p}{q(p-1)}}}C_{\frac{2N}{p}N, N}\|\nabla u\|_{L^{q}(\Omega)}.
\end{align*}
\end{proof}

\begin{theorem}\label{conesti00}
Let $\Omega\subset\mathbb{R}^N~(N\geq 2)$ be a bounded convex domain. Assume that $2<p<2N/(N-2)$ if $N\geq 3$ and $2<p<\infty$ if $N=2$. For all $u\in W^{1,2}(\Omega)$, we have
\begin{align*}
\|u-u_{\Omega}\|_{L^p(\Omega)}\leq D_p(\Omega)\|\nabla u\|_{L^{2}(\Omega)} 
\end{align*}
with
\begin{align*}
D_p(\Omega)=\frac{{d_\Omega}^{1+\frac{p+2}{2p}N}\pi^{\frac{p+2}{4p}N}}{N|\Omega|}\frac{\Gamma(\frac{p-2}{4p}N)}{\Gamma(\frac{p+2}{4p}N)}\sqrt{\frac{\Gamma(\frac{N}{p})}{\Gamma(\frac{p-1}{p}N)}}\left(\frac{\Gamma(N)}{\Gamma(\frac{N}{2})}\right)^{\frac{p-2}{p}}.
\end{align*} 
\end{theorem}

\begin{proof}
Let $u\in W^{1,2}(\Omega)$. Since $p<2N/(N-2)$ and $1-N+(p+2)N/(2p)>0$, it follows that $|x-z|^{1-N+(p+2)N/(2p)}\leq d_{\Omega}^{1-N+(p+2)N/(2p)}$ for $x, z\in\Omega$. Lemma \ref{eq1} leads to 
\begin{align*}\nonumber
|u(x)-u_{\Omega}|&\leq \frac{d_{\Omega}^{N}}{N|\Omega|}\int_{\Omega}|x-z|^{1-N+\frac{p+2}{2p}N}|x-z|^{-\frac{p+2}{2p}N}|\nabla u(z)|dz\\ \nonumber
&\leq\frac{d_{\Omega}^{1+\frac{p+2}{2p}N}}{N|\Omega|}\int_{\Omega}|x-z|^{-\frac{p+2}{2p}N}|\nabla u(z)|dz\\
&\leq\frac{d_{\Omega}^{1+\frac{p+2}{2p}N}}{N|\Omega|}\int_{\mathbb{R}^N}|x-z|^{-\frac{p+2}{2p}N}\left(E_{\Omega,\mathbb{R}^N}|\nabla u|\right)(z)dz.
\end{align*}
Therefore, 
\begin{align*}
\|u-u_{\Omega}\|_{L^p(\Omega)}&\leq \frac{{d_\Omega}^{1+\frac{p+2}{2p}N}}{N|\Omega|}\left(\int_{\Omega}\left(\int_{\mathbb{R}^N}|x-z|^{-\frac{p+2}{2p}N}\left(E_{\Omega,\mathbb{R}^N}|\nabla u|\right)(z)dz\right)^pdx\right)^{\frac{1}{p}}\\
&\leq\frac{{d_\Omega}^{1+\frac{p+2}{2p}N}}{N|\Omega|}\left(\int_{\mathbb{R}^N}\left(\int_{\mathbb{R}^N}|x-z|^{-\frac{p+2}{2p}N}\left(E_{\Omega,\mathbb{R}^N}|\nabla u|\right)(z)dz\right)^pdx\right)^{\frac{1}{p}}.
\end{align*}
From \eqref{hlss2}, it follows that
\begin{align*}
\|u-u_{\Omega}\|_{L^p(\Omega)}&\leq\frac{d_{\Omega}^{1+\frac{p+2}{2p}N}}{N|\Omega|}\tilde C_{\frac{p+2}{2p}N, N}\|E_{\Omega,\mathbb{R}^N}|\nabla u|\|_{L^{2}(\mathbb{R}^N)}\\
&=\frac{d_{\Omega}^{1+\frac{p+2}{2p}N}}{N|\Omega|}\tilde C_{\frac{p+2}{2p}N, N}\|\nabla u\|_{L^{2}(\Omega)},
\end{align*}
where $\tilde C_{\frac{p+2}{2p}N, N}$ is defined in \eqref{hlsesti2} with $\lambda=(p+2)N/(2p)$. 
\end{proof}

\begin{theorem}\label{conesti}
Let $\Omega\subset\mathbb{R}^N~(N\in\mathbb{N})$ be a bounded convex domain. Suppose that $1\leq q\leq p<qN/(N-q)$ if $N>q$, and $1\leq q\leq p<\infty$ if $N=q$. 
Then, we have
\begin{align}\label{result0}
\|u-u_{\Omega}\|_{L^p(\Omega)}\leq D_p(\Omega)\|\nabla u\|_{L^q(\Omega)}~\mbox{for all}~u\in W^{1,q}(\Omega) 
\end{align}
with
\begin{align*}
D_p(\Omega)=\frac{{d_\Omega}^N}{N|\Omega|}(A_{r}A_{q}A_{p'})^N \||x|^{1-N}\|_{L^{r}(V)},
\end{align*}
where $\Omega_{x}:=\{x-y\mid y\in\Omega\}$ for $x\in\Omega$, $V:=\cup_{x\in\Omega}\Omega_x$, and $r=qp/((q-1)p+q)$.
\end{theorem}

\begin{proof} 
First, we prove $I:=\||x|^{1-N}\|_{L^{r}(V)}^r<\infty$. Let  $\rho=2d_{\Omega}$ so that $V\subset B(0,\rho)$. We have   
\begin{align*}
\frac{pq(1-N)}{(q-1)p+q}+N-1&=\frac{pq(1-N)+Np(q-1)+Nq}{(q-1)p+q}-1\\
&=\frac{Nq-(N-q)p}{(q-1)p+q}-1>-1.
\end{align*}

Therefore, 
\begin{align*} 
I&=\int_{V}|x|^{\frac{pq(1-N)}{(q-1)p+q}}dx
\leq\int_{B(0,\rho)}|x|^{\frac{pq(1-N)}{(q-1)p+q}}dx
=J\int_{0}^{\rho}\rho^{\frac{pq(1-N)}{(q-1)p+q}+N-1}d\rho~<\infty,
\end{align*}
where $J$ is defined by
\[
J=
\begin{cases}
2&\text{($N=1$)},\\
2\pi&\text{($N=2$)},\\
2\pi\displaystyle\int_{[0,\pi]^{N-2}}\displaystyle\prod_{i=1}^{N-2}(\sin \theta_i)^{N-i-1} d\theta_1\cdots d\theta_{N-2}&\text{($N\geq 3$)}.
\end{cases}
\] 

Next, we show \eqref{result0}. For $x\in\Omega$,  it follows from Lemma \ref{eq1} that  
\begin{align*}
|u(x)-u_{\Omega}|&\leq\frac{d_{\Omega}^N}{N|\Omega|}\int_{\Omega}|x-y|^{1-N}|\nabla u(y)|dy\\ 
&=\frac{d_{\Omega}^N}{N|\Omega|}\int_{\Omega_x}|y|^{1-N}|\nabla u(x-y)|dy\\ 
&\leq\frac{d_{\Omega}^N}{N|\Omega|}\int_{V}|y|^{1-N}\left(E_{\Omega,V}|\nabla u|\right)(x-y)dy.
\end{align*}

Since $E_{V,\mathbb{R}^N}E_{\Omega,V}=E_{\Omega,\mathbb{R}^N}$, 
\begin{align}\label{kiso1}
|u(x)-u_{\Omega}|&\leq\frac{d_{\Omega}^N}{N|\Omega|}\int_{\mathbb{R}^N}\left(E_{V,\mathbb{R}^N}\psi\right)(y)\left(E_{\Omega,\mathbb{R}^N}|\nabla u|\right)(x-y)dy,
\end{align}
where $\psi(y)=|y|^{1-N}$ for $y\in V$. We denote $f(x)=\left(E_{V,\mathbb{R}^N}\psi\right)(x)$ and $g(x)=\left(E_{\Omega,\mathbb{R}^N}|\nabla u|\right)(x)$. Lemma \ref{th1} and \eqref{kiso1} give

\begin{align*}
\|u-u_{\Omega}\|_{L^p(\Omega)}&\leq\frac{d_{\Omega}^N}{N|\Omega|}\|f*g\|_{L^p(\Omega)}\\
&\leq\frac{d_{\Omega}^N}{N|\Omega|}\|f*g\|_{L^p(\mathbb{R}^N)}\\ 
&\leq \frac{d_{\Omega}^N}{N|\Omega|}(A_{r}A_{q}A_{p'})^N\|f\|_{L^{r}(\mathbb{R}^N)}\|g\|_{L^{q}(\mathbb{R}^N)}\\
&=\frac{d_{\Omega}^N}{N|\Omega|}(A_{r}A_{q}A_{p'})^N I^{\frac{1}{r}}\|\nabla u\|_{L^{q}(\Omega)}.
\end{align*}

\end{proof}

\begin{theorem}\label{conesti1}
Let $\Omega\subset\mathbb{R}^N~(N\in\mathbb{N})$ be a bounded convex domain, and let $q>N$. Then, we have
\begin{align}\label{result1}
\|u-u_{\Omega}\|_{L^\infty(\Omega)}\leq D_\infty(\Omega)\|\nabla u\|_{L^q(\Omega)}~\mbox{for all}~u\in W^{1,q}(\Omega)
\end{align}
with
\begin{align*}
D_\infty(\Omega)=\frac{{d_\Omega}^N}{N|\Omega|}\||x|^{1-N}\|_{L^{q'}(V)},
\end{align*}
where $V$ is defined in Theorem $\ref{conesti}$. 
\end{theorem}

\begin{proof}
First, we show $I:=\||x|^{1-N}\|_{L^{q'}(V)}^{q'}<\infty$. Let $\rho=2d_{\Omega}$ so that $V\subset B(0,\rho)$. We have
\begin{align*}
q'(1-N)+N-1=\frac{q(1-N)+N(q-1)}{q-1}-1=\frac{q-N}{q-1}-1>-1.
\end{align*} 
Therefore, 
\begin{align*} 
I=\int_{V}|x|^{q'(1-N)}dx\leq\int_{B(0,\rho)}|x|^{q'(1-N)}dx
=J\int_{0}^{\rho}\rho^{q'(1-N)+N-1}d\rho~<\infty,
\end{align*}
where $J$ is defined in the proof of Theorem $\ref{conesti}$.

Next, we prove \eqref{result1}. Let $r=\frac{q}{q-1}(\geq\hspace{-1.0mm}1)$, $f(x)=\left(E_{V,\mathbb{R}^N}\psi\right)(x)$, and $g(x)=\left(E_{\Omega,\mathbb{R}^N}|\nabla u|\right)(x)$, where $\psi$ is denoted in the proof of Theorem $\ref{conesti}$. From Lemma \ref{th1} and \eqref{kiso1}, it follows that
\begin{align*}
\|u-u_{\Omega}\|_{L^\infty(\Omega)}&\leq\frac{d_{\Omega}^N}{N|\Omega|}\|f*g\|_{L^\infty(\Omega)}
\leq\frac{d_{\Omega}^N}{N|\Omega|}\|f*g\|_{L^\infty(\mathbb{R}^N)}\\
&\leq\frac{d_{\Omega}^N}{N|\Omega|}\|f\|_{L^{q'}(\mathbb{R}^N)}\|g\|_{L^{q}(\mathbb{R}^N)}
=\frac{d_{\Omega}^N}{N|\Omega|}I^{\frac{1}{q'}}\|\nabla u\|_{L^{q}(\Omega)}.
\end{align*}

\end{proof}

\section{Estimation of $C_p(\Omega)$ for certain domains}\label{sec:num}
In this section, we present numerical examples where explicit values of $C_p(\Omega)$ on a square and a triangle domain are computed using Theorem \ref{mainembedding}, \ref{conesti0}, \ref{conesti00}, \ref{conesti}, and \ref{conesti1}.
All computations were performed on a computer with Intel Xeon E7-4830 v2 at 2.20 GHz$\times$4, 2 TB RAM, CentOS 6.6, and MATLAB 2016a.
All rounding errors were strictly estimated using toolbox the INTLAB version 9 \cite{Intlab} for verified numerical computations. 
Therefore, the accuracy of all results was mathematically guaranteed. 

First, we select domains $\Omega_i~(1\leq i\leq n)$ satisfying \eqref{omega1} and \eqref{omega2}. For all domains $\Omega_i~(1\leq i\leq n)$, we then compute the values of $D_p(\Omega_i)$ using Theorem \ref{conesti0}, \ref{conesti00}, \ref{conesti}, and \ref{conesti1}. Next, explicit values of $C_p(\Omega)$ are computed through Theorem \ref{mainembedding}. 

\subsection{Estimation on a square domain}
For the first example, we select the case in which $\Omega=(0,1)^2$. 
In this case, $V$ (in Theorem \ref{conesti} and \ref{conesti1}) becomes a square with side length $2/\sqrt{n}$ (see Fig.~\ref{fig:V1}).
Note that $\||x|^{1-N}\|_{L^r(V)}=\int_{V}|x|^{\beta}dx$, where $\beta=qp(1-N)/((q-1)p+q)$ if $p<\infty$ and $\beta=q'(1-N)$ if $p=\infty$.

For $n=1,4,16, 64, \cdots$, we define each $\Omega_i~(1\leq i\leq n)$ as a square with side length $1/\sqrt{n}$; see Fig.~\ref{fig:omegarec} for the cases in which $n=4$ and $n=16$.
For this division of $\Omega$, Theorem \ref{mainembedding} states that 
\begin{align*}
C_p(\Omega)=2^{1-\frac{1}{q}}\max\left(n^{-\left(\frac{1}{p}-\frac{1}{q}\right)},~\max_{1\leq i\leq n}D_p(\Omega_i)\right).
\end{align*}
Table \ref{tb:square} compares upper bounds for $C_p(\Omega)$ computed by Theorem \ref{conesti0}, \ref{conesti00}, \ref{conesti}, \cite[Lemma $2.3$]{cai}, and \cite[Corollary $4.3$]{tanaka};
the numbers of division $n$ are shown in the corresponding parentheses.
Moreover, these values
are plotted in Fig.~\ref{plotrec}, except for the values derived from \cite[Corollary $4.3$]{tanaka}.

Theorem \ref{conesti0}, \ref{conesti00}, \ref{conesti}, and \cite[Lemma $2.3$]{cai} provide sharper estimates of $C_p(\Omega)$ than \cite[Corollary $4.3$]{tanaka} for all $p$'s.
The estimates derived by Theorem \ref{conesti00} and Theorem \ref{conesti} for $31\leq p\leq 80$ are sharper than the estimates obtained by \cite[Lemma $2.3$]{cai}.

We also show the values of $C_\infty(\Omega)$ computed by Theorem \ref{conesti1} for $3\leq q\leq10$ in Table \ref{tb:square1}.
\begin{figure}[h]
\begin{minipage}{0.5\hsize}
\centering
\includegraphics[width=7.0cm]{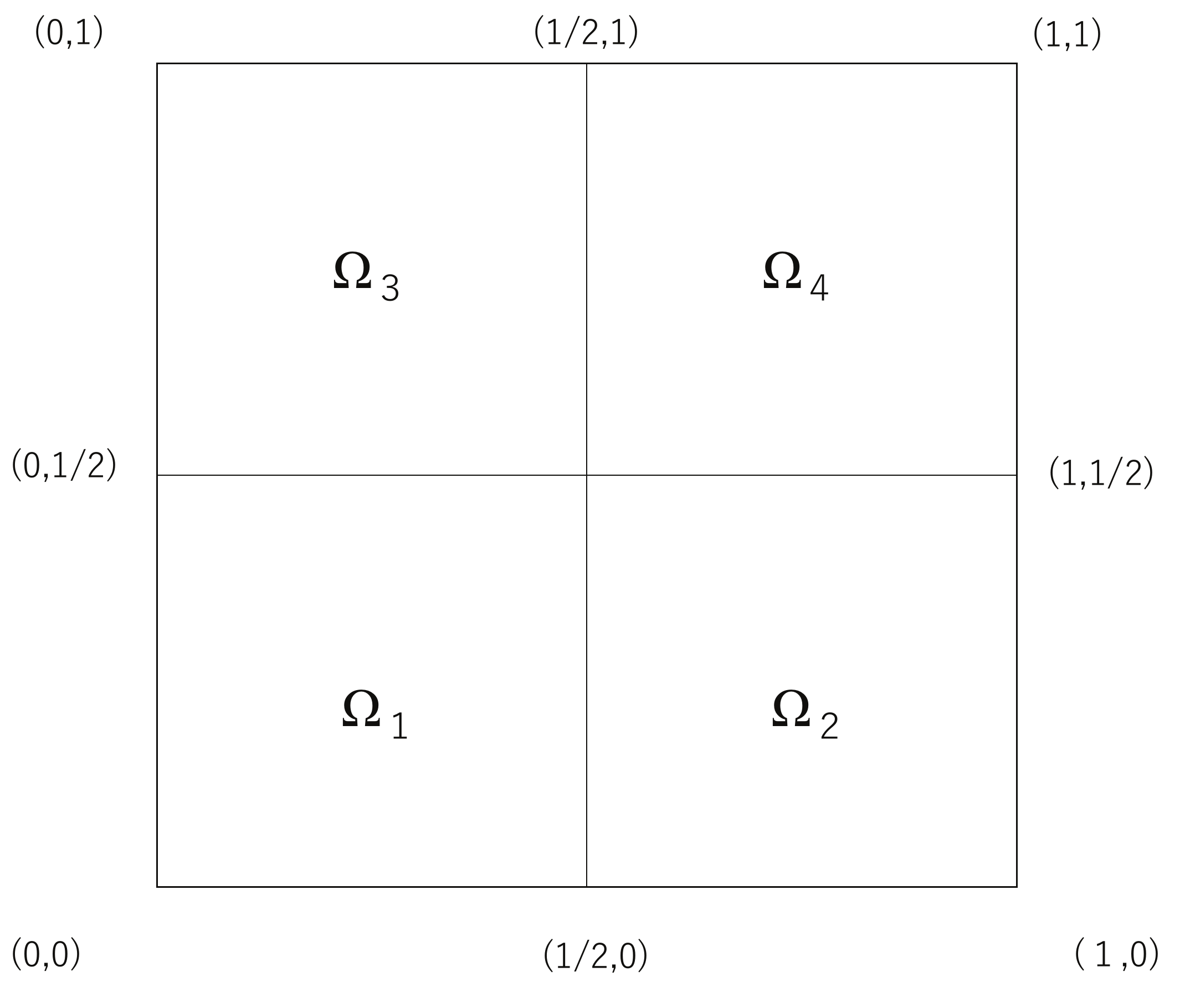}
\end{minipage}
\begin{minipage}{0.5\hsize}
\centering
\includegraphics[width=7.0cm]{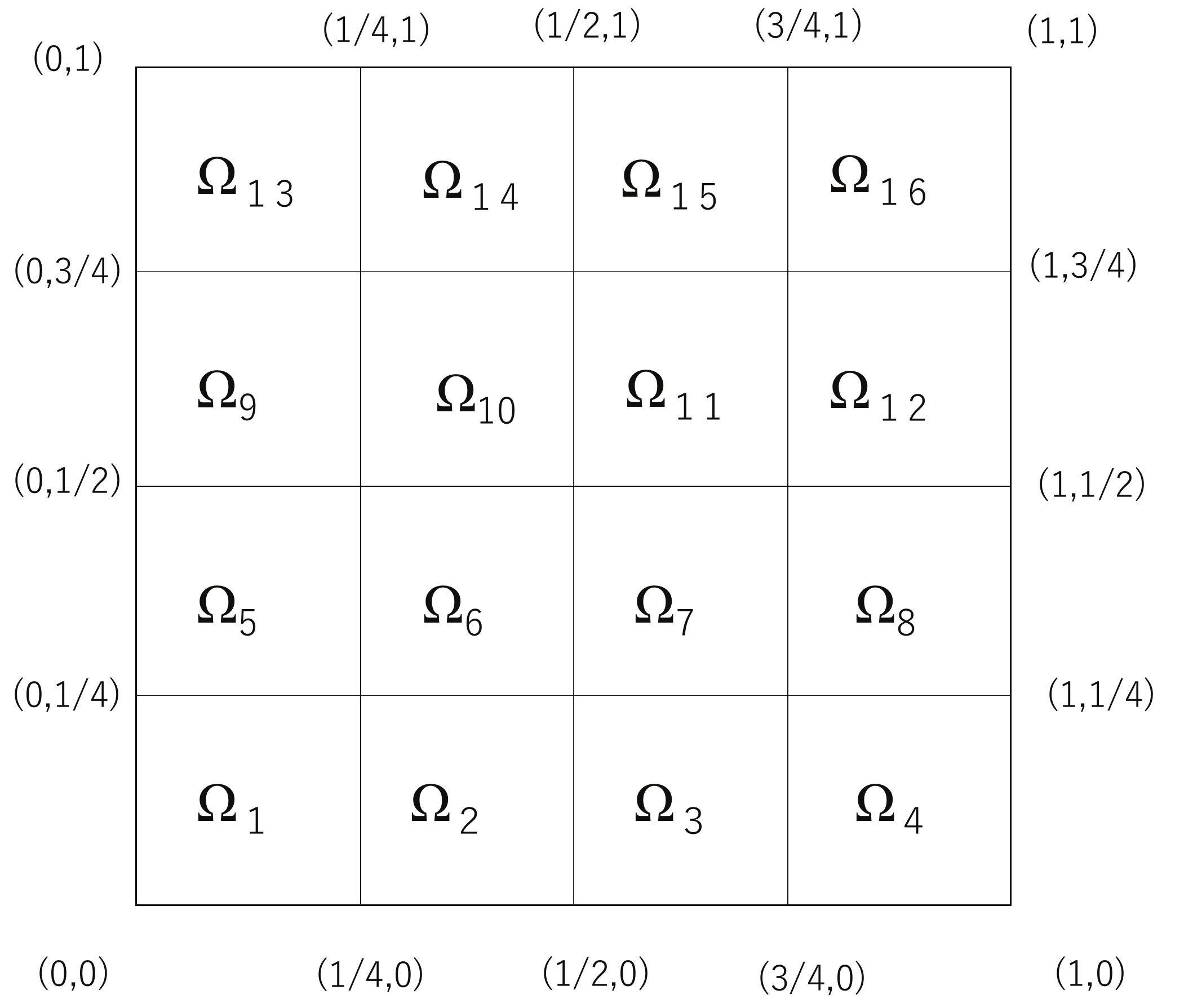}
\end{minipage}
\vspace{0.3cm}\caption{$\Omega_i$ for the cases in which $n=4$ (the left side) and  $n=16$ (the right side).}
\label{fig:omegarec}
\end{figure}

\begin{figure}
\centering
\includegraphics[width=8.5cm]{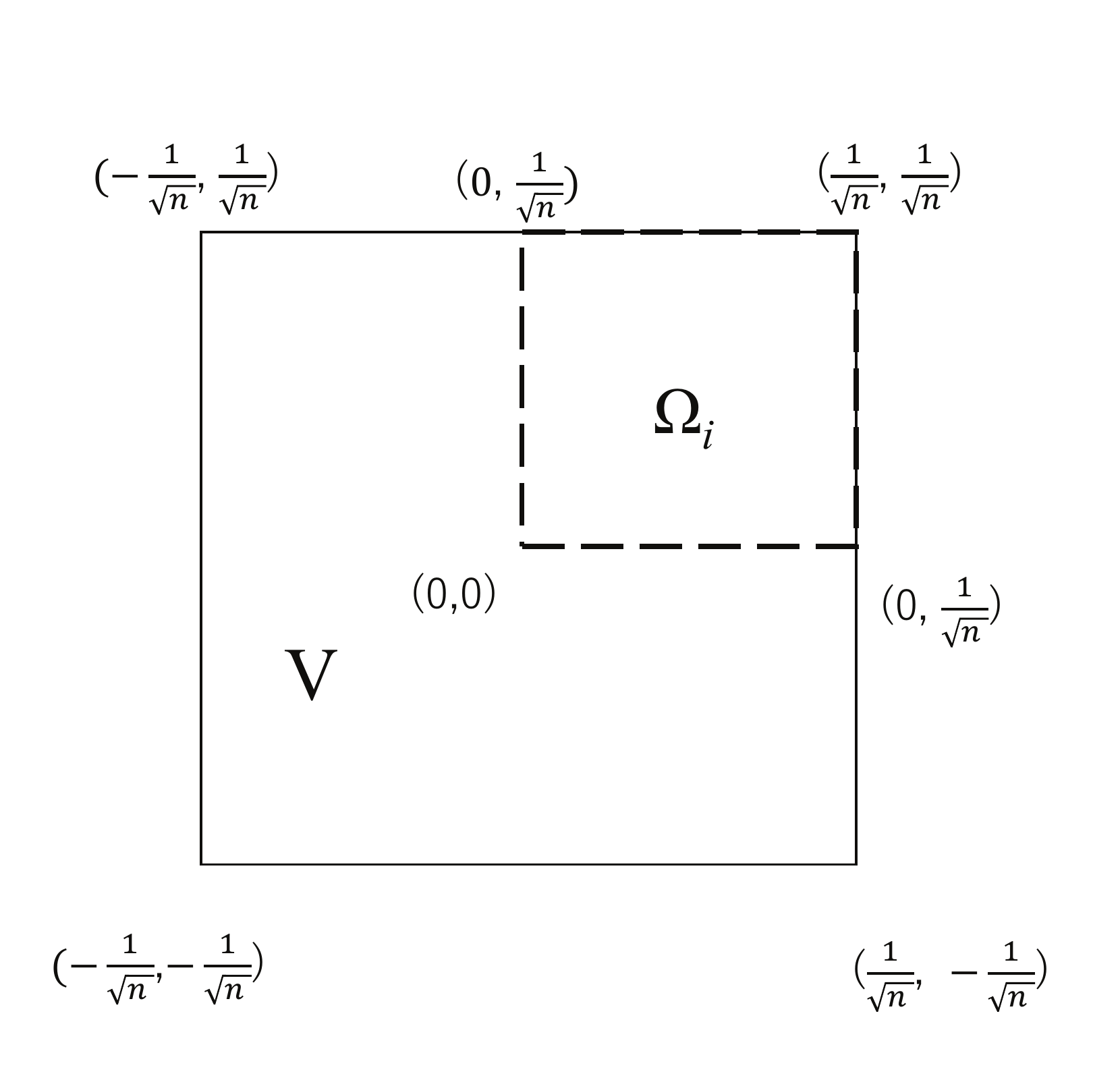}
\caption{The domain $V$ in Theorem \ref{conesti} and \ref{conesti1}.}
\label{fig:V1}
\end{figure}

\begin{table}[h]
\caption{Computed values of $C_p(\Omega)$ for $\Omega=(0,1)^2$ and $q=2$.
The numbers of division $n$ are shown in the corresponding parentheses.
Theorem \ref{conesti0} cannot be used for $p>4$ when $N=2$.}
\begin{center}
{\renewcommand\arraystretch{1.1}
{\tabcolsep=1.2mm
\begin{tabular}{c||ccccc}
\hline
$p$ & Theorem \ref{conesti0} & Theorem \ref{conesti00} & Theorem \ref{conesti} & \cite[Lemma $2.3$]{cai} & \cite[Corollary $4.3$]{tanaka}\\
\hline
3&2.828428(64)&21.041858(1)&2.6470760(16)&1.272533&1.291703$\times 10^4$\\ 
4&2.828428(16)&12.804451(1)&3.0989536(16)&1.553774&1.809271$\times 10^4$\\ 
5&-&10.313751(1)&3.527578(16)&1.841950&2.275458$\times 10^4$\\ 
6&-&9.210466(1)&3.922709(16)&2.135792&2.701890$\times 10^4$\\ 
7&-&8.643432(1)&4.288114(16)&2.434362&3.096661$\times 10^4$\\ 
8&-&8.335480(1)&4.628497(16)&2.736941&3.465528$\times 10^4$\\ 
9&-&8.170423(1)&4.947849(16)&3.042967&3.812726$\times 10^4$\\ 
10&-&8.091385(1)&5.249352(16)&3.351991&4.141471$\times 10^4$\\ 
\hline
20&-&8.698248(1)&7.659208(16)&6.549949&6.789009$\times 10^4$\\ 
30&-&9.741473(1)&9.485455(16)&9.856546&8.800592$\times 10^4$\\ 
40&-&10.75962(1)&10.640059(64)&13.218367&1.048141$\times 10^5$\\ 
50&-&11.71416(1)&12.020066(64)&16.613831&1.195208$\times 10^5$\\ 
60&-&12.60732(1)&13.258962(64)&20.031993&1.327453$\times 10^5$\\ 
70&-&13.44678(1)&14.392550(64)&23.466517&1.448540$\times 10^5$\\ 
80&-&14.23999(1)&15.443710(64)&26.913400&1.560849$\times 10^5$\\ 
\hline
\end{tabular}
\label{tb:square}
}}
\end{center}
\end{table}

\begin{table}[h]
\caption{Computed values of $C_\infty(\Omega)$ for $\Omega=(0,1)^2$ and $3\leq q\leq10$.
The numbers of division $n$ are shown in the corresponding parentheses.}
\begin{center}
{\renewcommand\arraystretch{1.1}
{\tabcolsep=1.2mm
\begin{tabular}{c||c}
\hline
$q$ & Theorem \ref{conesti1}\\
\hline
3&5.611920(16)\\ 
4&4.756829(64)\\ 
5&4.000001(64)\\ 
6&3.563595(64)\\ 
7&3.281342(64)\\ 
8&3.084422(64)\\ 
9&2.939469(64)\\ 
10&2.828428(64)\\ 
\hline
\end{tabular}
\label{tb:square1}
}}
\end{center}
\end{table}

\begin{figure}[H]
\centering
\includegraphics[width=12.0cm]{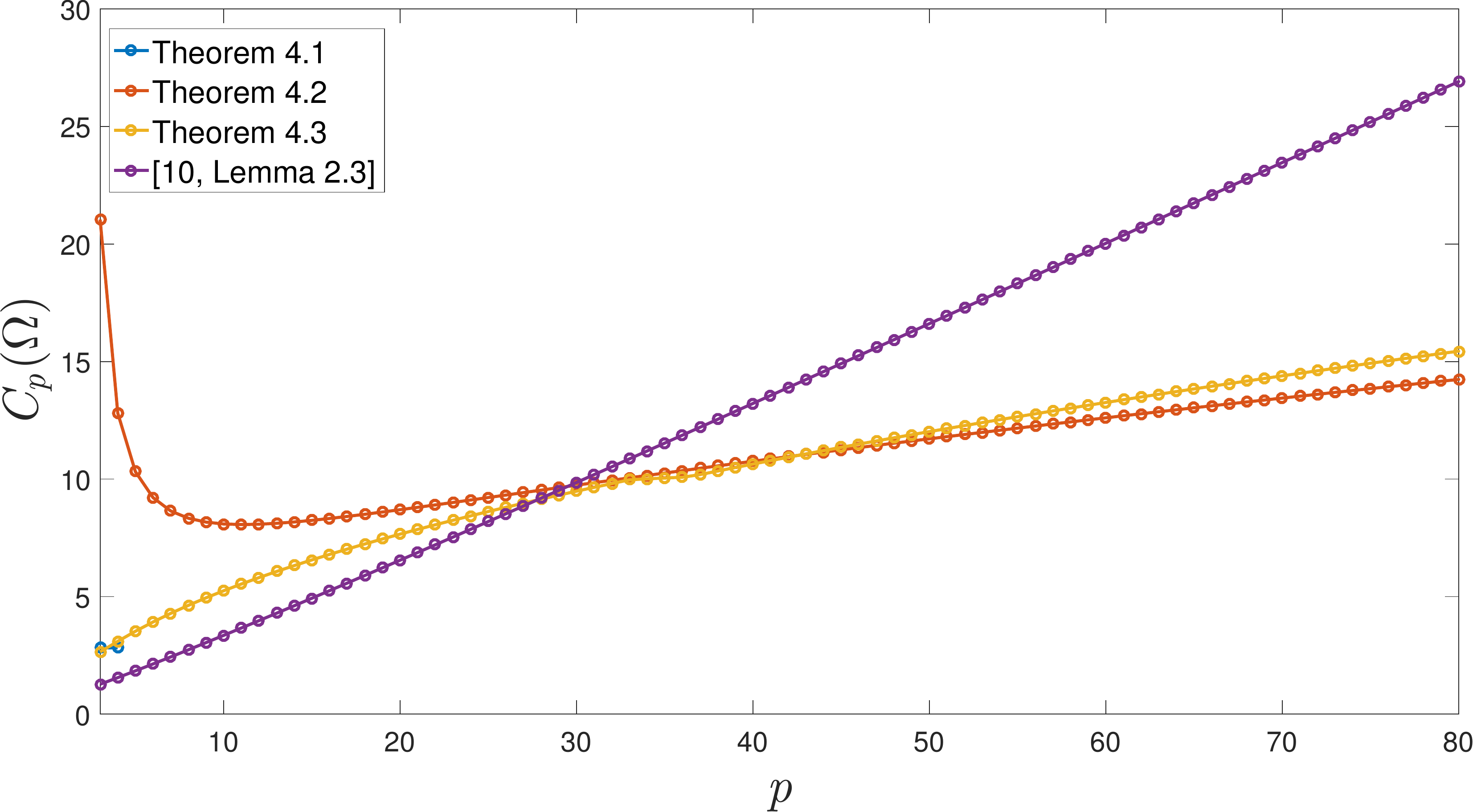}
\vspace{0.3cm}\caption{Computed values of $C_p(\Omega)$ for $\Omega=(0,1)^2$ and $3\leq p\leq 80$.}
\label{plotrec}
\end{figure}


\subsection{Estimation on a triangle domain} 
For the second example, we select the case in which $\Omega$ is a regular triangle with the vertices $(0,0)$, $(1,0)$, and $(1/2,\sqrt{3}/2)$.
In this case, $V$ is the regular hexagon displayed in Fig.~\ref{fig:V2}.

For $n=1,4,16,64,\cdots$, we define each $\Omega_i~(1\leq i\leq n)$ as a regular triangle with side length $1/\sqrt{n}$; see Fig.~\ref{fig:omegatri} for the case in which $n=4$ and $n=16$.
For this division of $\Omega$, Theorem \ref{mainembedding} states that
\[
C_p(\Omega)=2^{1-\frac{1}{q}}\max\left((4n)^{-\left(\frac{1}{p}-\frac{1}{q}\right)},~\max_{1\leq i\leq n} D_p(\Omega_i)\right).
\]
Table \ref{tb:triangle} compares upper bounds of $C_p(\Omega)$ computed by Theorem \ref{conesti0}, \ref{conesti00}, \ref{conesti}, and \cite[Corollary $4.3$]{tanaka};
the numbers of division $n$ are shown in the corresponding parentheses.
Moreover, these values are plotted in Fig.~\ref{fig:triplot}.
Theorem \ref{conesti} provides the sharpest estimates for all $3\leq p\leq 80$.

We also show the values of $C_\infty(\Omega)$ computed by Theorem \ref{conesti1} for $3\leq q\leq 10$ in Table \ref{tb:triangle1}.

\begin{figure}[H]
\begin{minipage}{0.5\hsize}
\centering
\includegraphics[width=6.8cm]{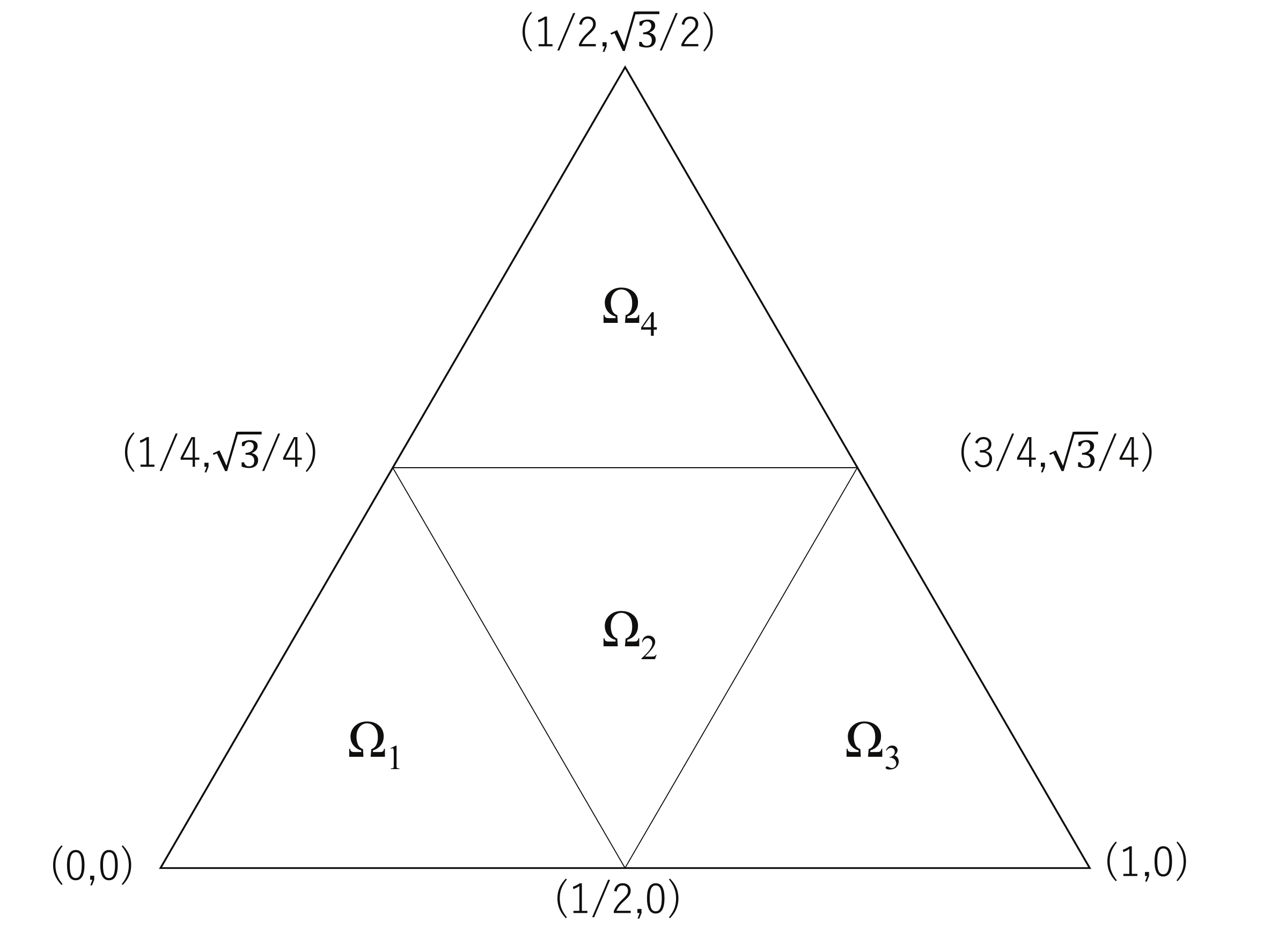}
\end{minipage}
\begin{minipage}{0.5\hsize}
\centering
\includegraphics[width=10.2cm]{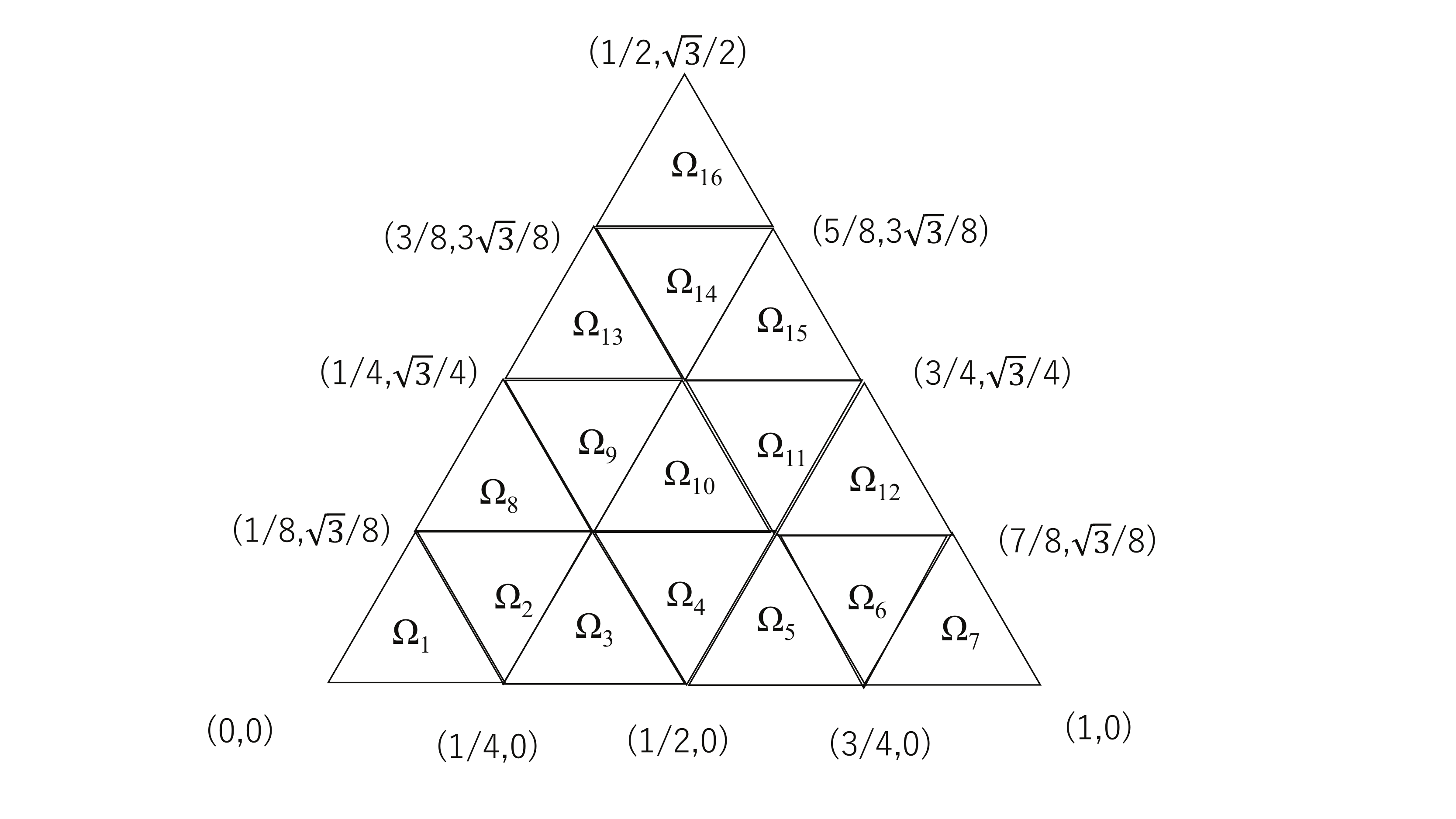}
\end{minipage}
\caption{$\Omega_i$ when $n=4$ (the left side) and $n=16$ (the right side).}
\label{fig:omegatri}
\end{figure}

\begin{figure}[H]
\centering
\includegraphics[width=9.5cm]{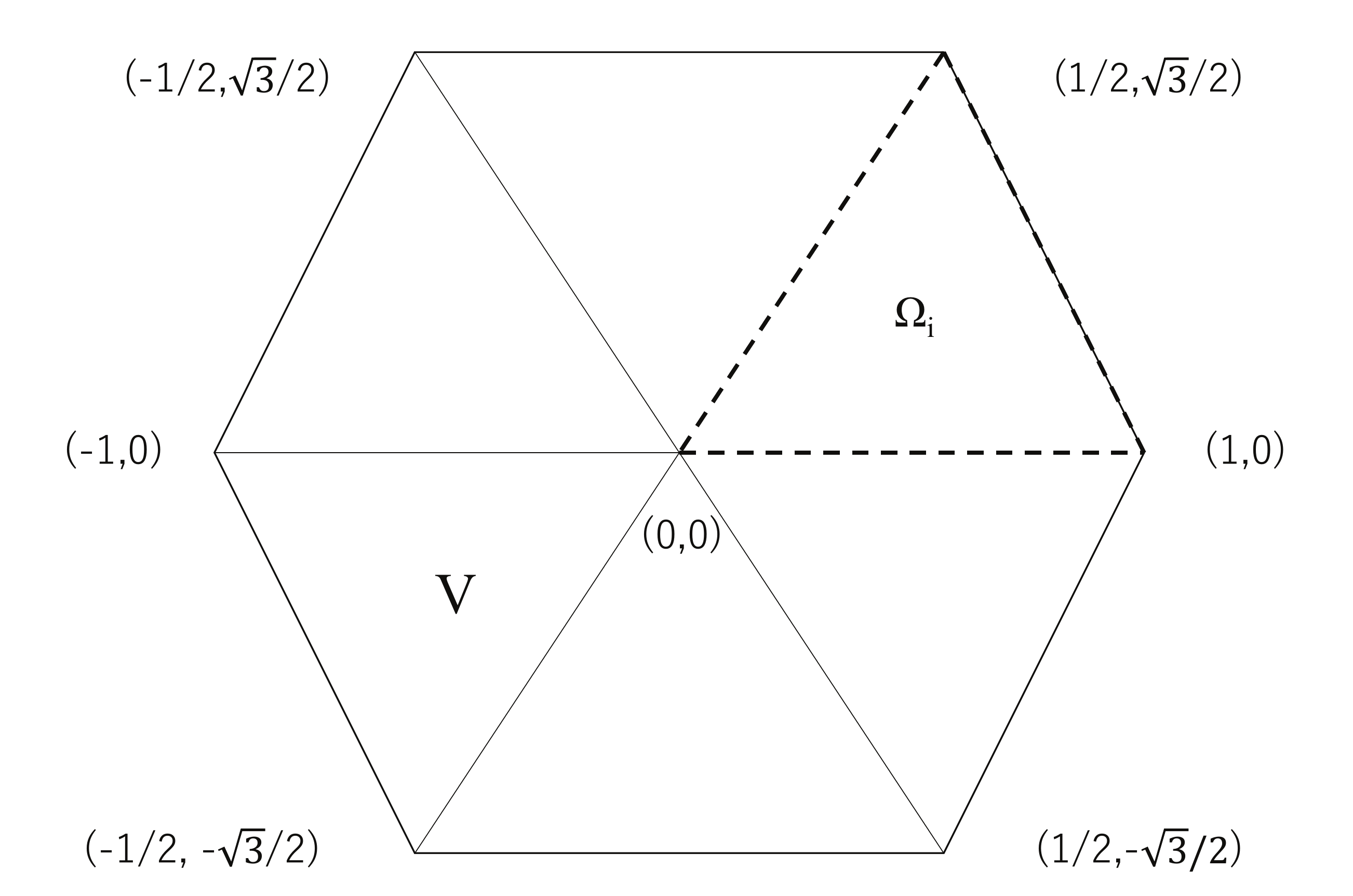}
\caption{The domain $V$ in Theorem \ref{conesti} and \ref{conesti1}.}
\label{fig:V2}
\end{figure}

\begin{table}[H]
\caption{Computed values of $C_p(\Omega)$ for a regular triangle domain $\Omega$ and $q=2$.
The numbers of division $n$ are shown in the corresponding parentheses.
Theorem \ref{conesti0} cannot be used for $p>4$ when $N=2$.}
\begin{center}
{\renewcommand\arraystretch{1.1}
{\tabcolsep=1.2mm
\begin{tabular}{c||cccc}
\hline
$p$ & Theorem \ref{conesti0} & Theorem \ref{conesti00} & Theorem \ref{conesti}& \cite[Corollary $4.3$]{tanaka}\\
\hline
3&3.251833(64)&25.741822(1)&2.366856(4)&2.538335$\times 10^4$\\ 
4&3.320470(4)&16.123490(1)&2.709475(4)&3.553398$\times 10^4$\\ 
5&-&13.214188(1)&3.042818(4)&4.464990$\times 10^4$\\ 
6&-&11.937755(1)&3.353176(4)&5.297547$\times 10^4$\\ 
7&-&11.295642(1)&3.641844(4)&6.067602$\times 10^4$\\ 
8&-&10.960821(1)&3.911816(4)&6.786738$\times 10^4$\\ 
9&-&10.795618(1)&4.165864(4)&7.463399$\times 10^4$\\ 
10&-&10.732444(1)&4.406282(4)&8.103954$\times 10^4$\\
\hline
20&-&11.739049(1)&6.341217(4)&1.326097$\times 10^5$\\
30&-&13.223132(1)&7.622031(16)&1.717928$\times 10^5$\\
40&-&14.647402(1)&8.748299(16)&2.045371$\times 10^5$\\
50&-&15.974507(1)&9.869218(16)&2.331904$\times 10^5$\\
60&-&17.212379(1)&10.876336(16)&2.589578$\times 10^5$\\
70&-&18.373623(1)&11.798394(16)&2.825529$\times 10^5$\\
80&-&19.469505(1)&12.653794(16)&3.044383$\times 10^5$\\
\hline
\end{tabular}
\label{tb:triangle}
}}
\end{center}
\end{table}

\begin{table}[H]
\caption{Computed values of $C_\infty(\Omega)$ for a regular triangle domain $\Omega$ and $3\leq q\leq 10$.
The numbers of division $n$ are shown in the corresponding parentheses.}
\begin{center}
{\renewcommand\arraystretch{1.1}
{\tabcolsep=1.2mm
\begin{tabular}{c||c}
\hline
$q$ & Theorem \ref{conesti1}\\
\hline
3&4.797132(4)\\ 
4&4.146459(16)\\ 
5&3.583834(16)\\ 
6&3.251833(16)\\ 
7&3.033691(16)\\ 
8&2.879743(16)\\ 
9&2.765427(16)\\ 
10&2.677251(16)\\ 
\hline
\end{tabular}
\label{tb:triangle1}
}}
\end{center}
\end{table}

\begin{figure}[h]
\centering
\includegraphics[width=12.0cm]{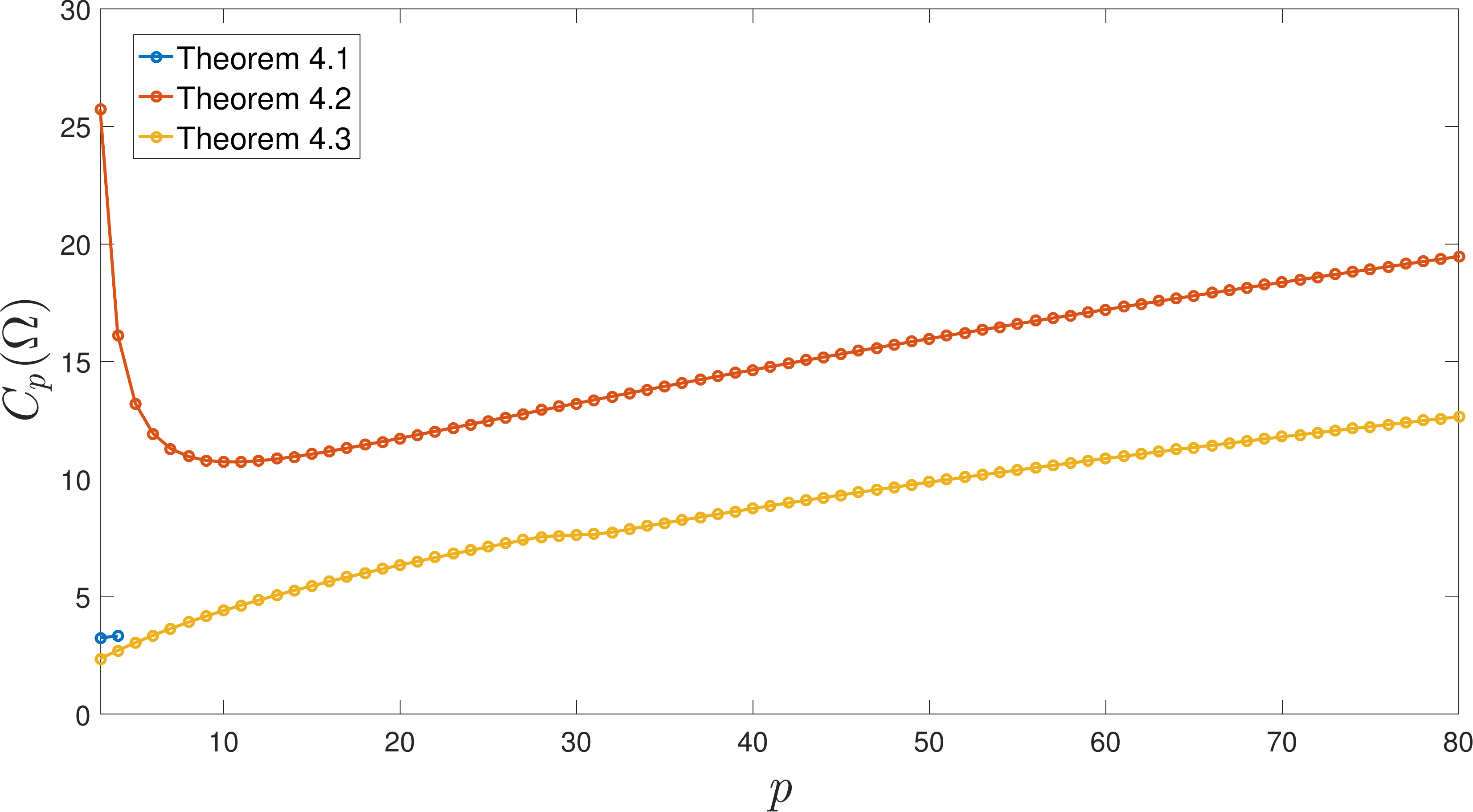}
\vspace{0.3cm}\caption{Computed values of $C_p(\Omega)$ for a regular triangle domain $\Omega$ and $3\leq p\leq 80$.}
\label{fig:triplot}
\end{figure}

\begin{remark}
The values of $C_p(\Omega)$ derived from Theorem $\ref{conesti0}$ to $\ref{conesti1}$ (provided in the Tables $\ref{tb:assume0}$ to $\ref{tb:triangle1}$) can be directly used for any domain that is composed of unit squares and triangles with side length $1$ $($see Fig.~$\ref{fig:mixdomain}$ for some examples$)$.
\end{remark}

\begin{figure}[h]
\begin{minipage}{0.5\hsize}
\centering
\includegraphics[width=8.3cm]{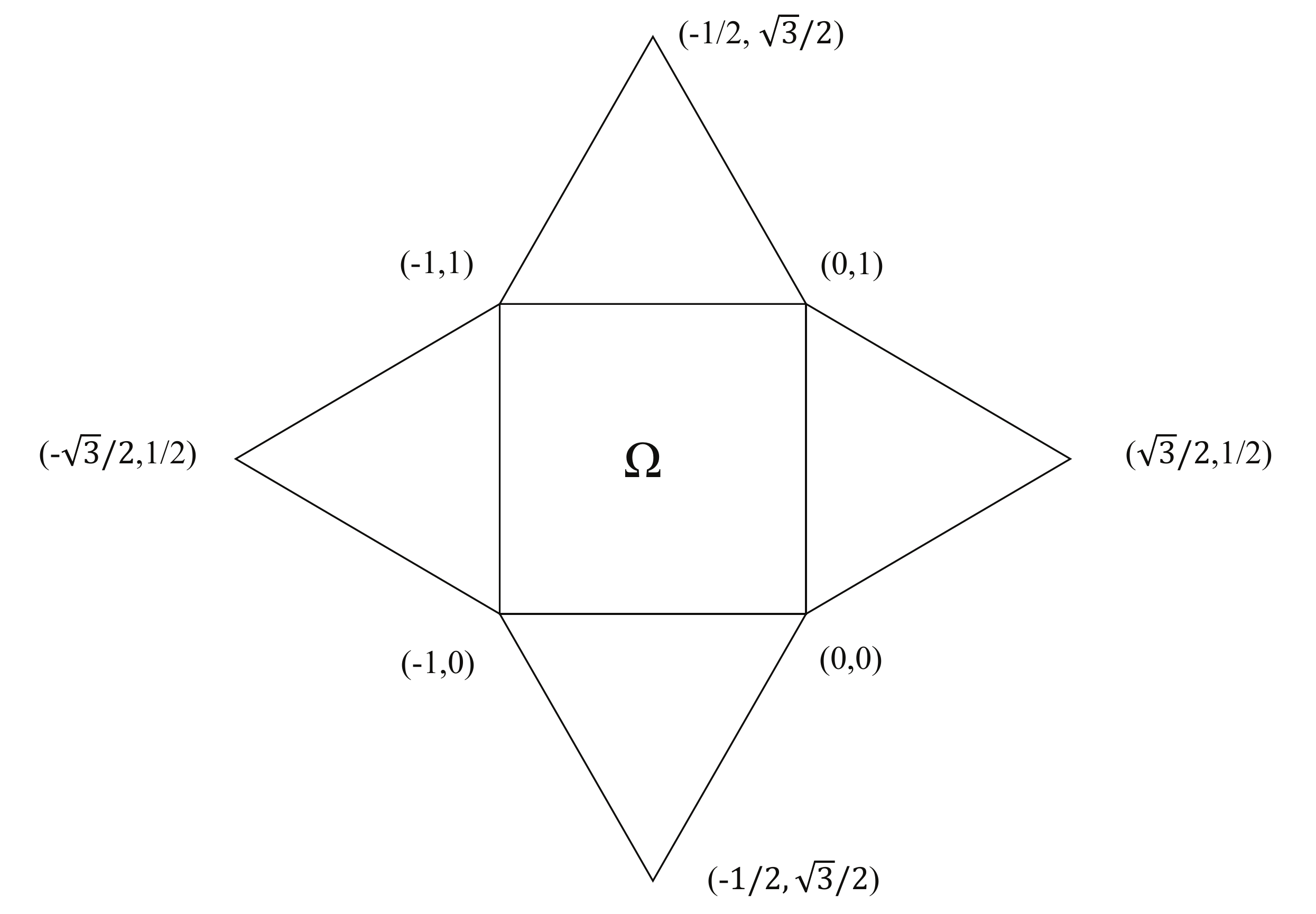}
\end{minipage}
\begin{minipage}{0.5\hsize}
\centering
\includegraphics[width=8.3cm]{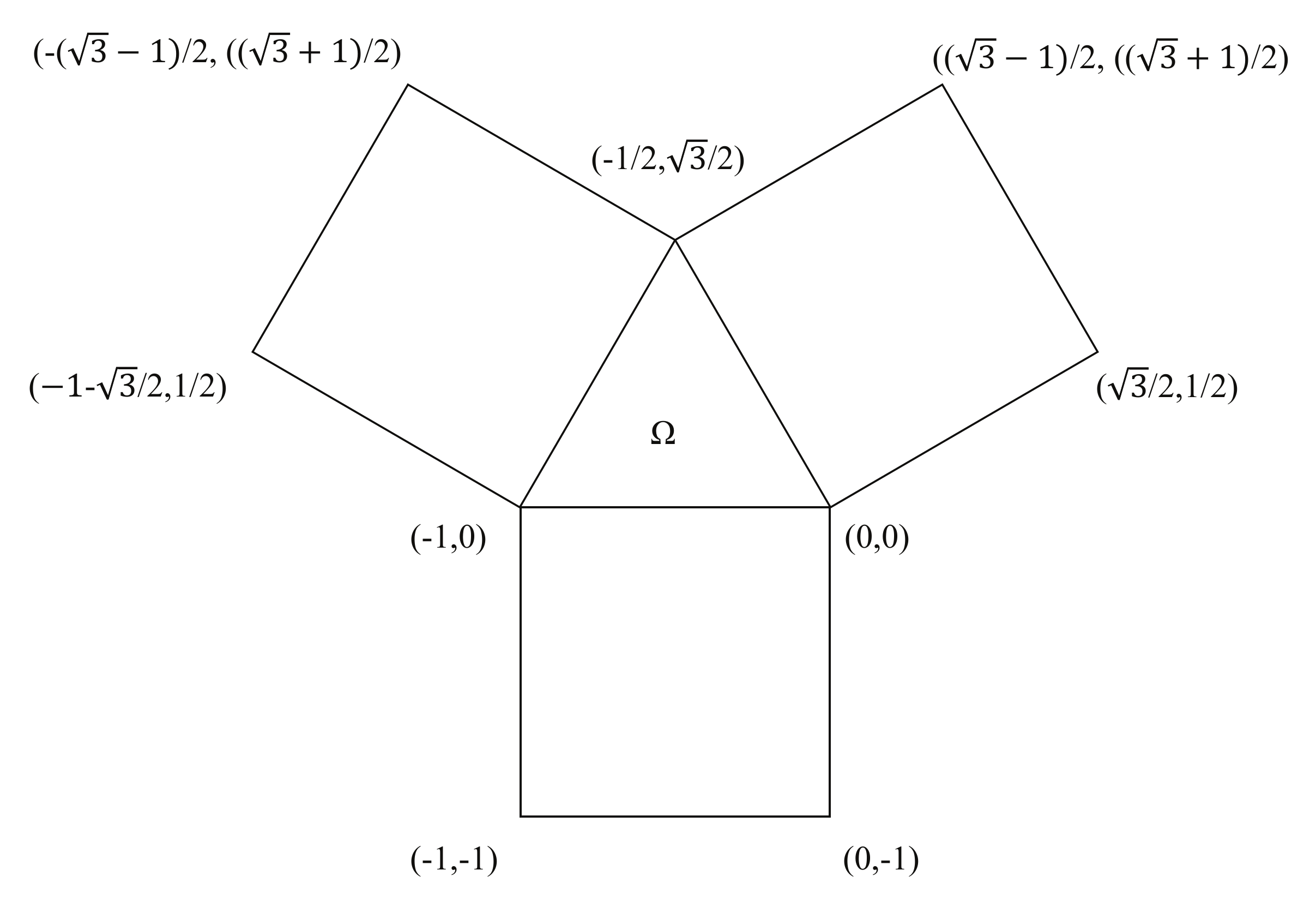}
\end{minipage}

\begin{minipage}{\hsize}
\centering
\includegraphics[width=7.5cm]{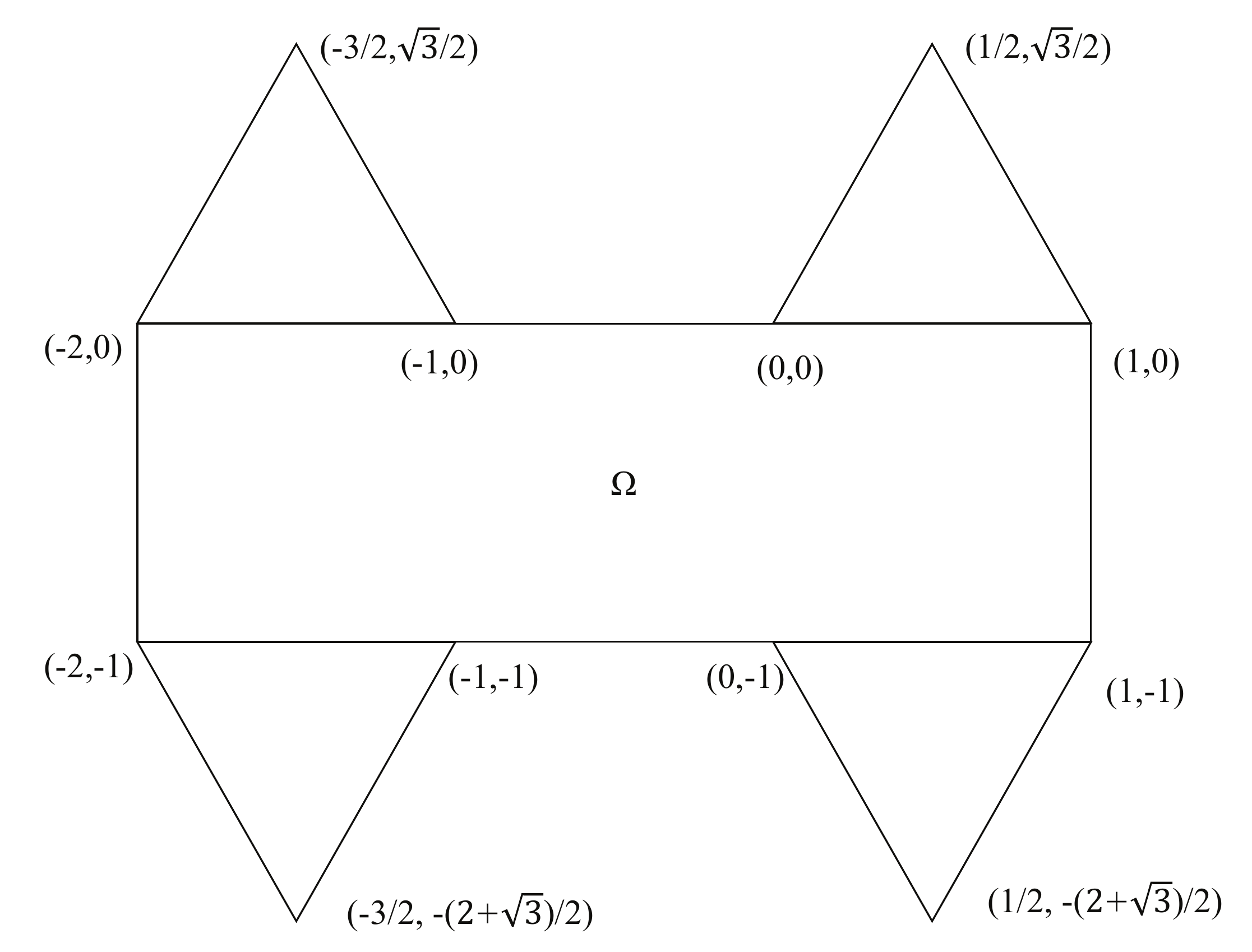}
\end{minipage}
\vspace{0.3cm}\caption{Examples of domains $\Omega$ that are composed of unit squares and triangles with side length $1$.}
\label{fig:mixdomain}
\end{figure}

\clearpage

\section{Conclusion}
We proposed several theorems that provide explicit values of Sobolev type embedding constant $C_p(\Omega)$ satisfying \eqref{aim} for a domain $\Omega$ that can be divided into a finite number of bounded convex domains.
These theorem give sharper estimates of $C_p(\Omega)$ than previous estimates derived by the method in \cite{tanaka}.
This accuracy improvement leads to much applicability of the estimates of $C_p(\Omega)$ to verified numerical computations for PDEs.
\appendix

\section{Embedding constant $C_p(\Omega)$ on dividable domains}
Theorem \ref{estimatecor} provides an estimation of the embedding constant $C_p(\Omega)$ for a domain $\Omega$ that can be divided into domains $\Omega_i$ (such as convex domains and Lipschitz domains) satisfying \eqref{omega1} and \eqref{omega2}.
\begin{theorem}\label{estimatecor}
Let $\Omega\subset\mathbb{R}^N~(N\in\mathbb{N})$ be a domain that can be divided into a finite number of domains $\Omega_i~(i=1,2,3,\cdots, n)$ satisfying \eqref{omega1} and \eqref{omega2}.
Assume that, for every $\Omega_i~(i=1,2,3,\cdots, n)$, there exists a constant $C_p(\Omega_i)$ such that $\|u\|_{L^p(\Omega_i)}\leq C_{p}(\Omega_i)\|u\|_{W^{1,q}(\Omega_i)}$ for all $u\in W^{1,q}(\Omega_i)$. Then, \eqref{aim} holds valid for 
\[
C_p(\Omega)=M_{p,q}\max_{1\leq i\leq n} C_{p}(\Omega_i),
\]
where  
\begin{align*}
M_{p,q}=
\begin{cases}
1\hspace{1cm}(p\geq q),\\[2mm]
n^{\frac{1}{p}-\frac{1}{q}}~~(p<q).
\end{cases}
\end{align*}

\end{theorem}

\begin{proof}
We consider both the cases in which $p<\infty$ and $p=\infty$.

When $p<\infty$, it follows that 
\begin{align*}\nonumber
\|u\|_{L^p(\Omega)}&=\left(\sum_{1\leq i\leq n}\|u\|_{L^p(\Omega_i)}^p\right)^{1/p}\\ \nonumber
&\leq\left(\sum_{1\leq i\leq n}C_{p}(\Omega_i)^p\|u\|_{W^{1,q}(\Omega_i)}^p\right)^{1/p}\\ \nonumber 
&\leq \max_{1\leq i\leq n} C_{p}(\Omega_i) \left(\sum_{1\leq i\leq n}\|u\|_{W^{1,q}(\Omega_i)}^p\right)^{1/p}\\
&\leq M_{p,q}\max_{1\leq i\leq n} C_{p}(\Omega_i) \|u\|_{W^{1,q}(\Omega)}.
\end{align*}
Note that $|x|_p\leq M_{p,q}|x|_q$ holds for $x=(x_1,x_2,\cdots,x_n)\in \mathbb{R}^n$ (see \cite[Lemma $A.1$]{tanaka} for a detailed proof), where we denote
\begin{align*}
|x|_p=
\begin{cases}
\left(\displaystyle\sum_{1\leq i\leq n}|x_i|^p\right)^{\frac{1}{p}}~~(1\leq p<\infty),\\[2mm]
\displaystyle\max_{1\leq i\leq n}|x_i|\hspace{1.5cm}(p=\infty).
\end{cases}
\end{align*}
When $p=\infty$, 
\begin{align*}\nonumber
\|u\|_{L^\infty(\Omega)}&= \max_{1\leq i\leq n}\|u\|_{L^{\infty}(\Omega_i)}\\ \nonumber
&\leq \max_{1\leq i\leq n} C_{p}(\Omega_i)\|u\|_{W^{1,q}(\Omega_i)}\\ \nonumber
&\leq \max_{1\leq i\leq n} C_{p}(\Omega_i)\max_{1\leq i\leq n}\|u\|_{W^{1,q}(\Omega_i)}.
\end{align*}

Since $M_{\infty,q}=1$, we have
\begin{align*}
\|u\|_{L^\infty(\Omega)}&\leq \max_{1\leq i\leq n} C_{p}(\Omega_i)\|u\|_{W^{1,q}(\Omega)}.
\end{align*}

\end{proof}

\section*{Acknowledgements}
This work was supported by CREST, Japan Science and Technology Agency. 
The second author (K.T.) was supported by the Waseda Research Institute for Science and Engineering, the Grant-in-Aid for Young Scientists (Early Bird Program). The third author (K.S.) was supported by JSPS KAKENHI Grant Number 16K17651.

\bibliographystyle{plain} 
\bibliography{ref}

\end{document}